\documentclass{elsarticle}[11pt]

\usepackage{amssymb,amsmath,amsfonts,amsthm}
\usepackage{relsize,exscale}
\usepackage{tikz-cd}
\usepackage{enumitem}
\usepackage{makeidx}
\usepackage{url}
\usepackage{xcolor}
 \DeclareFontFamily{U}{wncy}{}
    \DeclareFontShape{U}{wncy}{m}{n}{<->wncyr10}{}
    \DeclareSymbolFont{mcy}{U}{wncy}{m}{n}
    \DeclareMathSymbol{\Sha}{\mathord}{mcy}{"58} 
\makeindex
\usepackage{geometry}
\usepackage{amsmath}

\newtheorem{theorem}{Theorem}[section]
\newtheorem{corollary}[theorem]{Corollary}
\newtheorem{lemma}[theorem]{Lemma}
\newtheorem{proposition}[theorem]{Proposition}

\theoremstyle{definition}
\newtheorem{definition}[theorem]{Definition}
\newtheorem{remark}[theorem]{Remark}

\newtheorem{example}[theorem]{Example}

\begin{document}

\begin{frontmatter}

\title{The finitude of tamely ramified pro-$p$ extensions of number fields with cyclic $p$-class groups}

\author[inst1]{Yoonjin Lee}
\affiliation[inst1]{organization={Department of Mathematics, Ewha Womans University},
            city={Seoul},
            country={Republic of Korea}}
\ead{yoonjinl@ewha.ac.kr}

\author[inst2]{Donghyeok Lim \corref{cor1}}

\affiliation[inst2]{organization={Institute of Mathematical Sciences,
Ewha Womans University}, city={Seoul}, country={Republic of Korea}}
\ead{donghyeokklim@gmail.com}
\cortext[cor1]{Corresponding author}

\begin{abstract} 
 Let $p$ be an odd prime and $F$ be a number field whose $p$-class group is cyclic. Let $F_{\{\mathfrak{q}\}}$ be the maximal pro-$p$ extension of $F$ which is unramified outside a single non-$p$-adic prime ideal $\mathfrak{q}$ of $F$. In this work, we study the finitude of the Galois group $G_{\{\mathfrak{q}\}}(F)$ of $F_{\{\mathfrak{q}\}}$ over $F$.
 We prove that $G_{\{\mathfrak{q}\}}(F)$ is finite for the majority of $\mathfrak{q}$'s such that the generator rank of $G_{\{\mathfrak{q}\}}(F)$ is two, provided that for $p = 3$, $F$ is not a complex quartic field containing the primitive third roots of unity.
\end{abstract}


\end{frontmatter}

\date{}

\section{Introduction}

Let $p$ be a prime. Let $K$ be a number field, and let $S$ be a finite set of places of $K$. Let $K_{S}$ be the maximal pro-$p$ extension of $K$ unramified outside $S$. Let $G_{S}(K)$ be the Galois group of $K_{S}$ over $K$. It has been a long-standing problem to determine whether the Galois group $G_{S}(K)$ is finite or not. The problem is a generalization of the famous $p$-class field tower problem. If $S$ contains some primes of $K$ over $p$, then the abelianization $G_{S}(K)^{\mathrm{ab}}$ of $G_{S}(K)$ can be infinite. Hence, the class field theory can be used to prove the infinitude of $G_{S}(K)^{\mathrm{ab}}$ \cite{Maire4}. However, if $S$ consists only of non-$p$-adic places of $K$, then $G_{S}(K)^{\mathrm{ab}}$ is always finite. This so-called \textit{tame} case has been poorly understood so far. In this work, we exclusively study the finitude of the Galois group $G_{S}(K)$ in the tame case.

The principal method in the tame case is the theorem of Golod and Shafarevich \cite{GolodShafarevich}, \cite[\S 7.7]{Koch}. For a pro-$p$ group $G$, let $d(G)$ be the generator rank of $G$ and $r(G)$ be the relation rank of $G$. The theorem of Golod and Shafarevich states that $G$ is infinite if $d(G)^{2}/4 \geq r(G)$.  For a finite abelian group $\mathfrak{A}$, let $\mathrm{rk}_{p}(\mathfrak{A})$ be the $p$-rank of $\mathfrak{A}$. For a number field $K$, let us denote the $\mathbb{Z}$-rank of the multiplicative group of units of the ring $\mathcal{O}_K$ of integers of $K$ by $r_{K}$. We define $\theta_{K,S} := 1$ if $S$ is empty and $K$ contains the primitive $p$th roots of unity and $\theta_{K,S} := 0$ otherwise. For the arithmetic pro-$p$ groups $G_{S}(K)$, their invariants $d(G_{S}(K))$ and $r(G_{S}(K))$ have been studied in terms of the arithmetic of $K$. By the theorem of Golod-Shafarevich, when $S$ is a finite set of finite non-$p$-adic primes of $K$, the group $G_{S}(K)$ is infinite if 
 \begin{equation}\label{eq1}
     \mathrm{rk}_{p}(\mathrm{Cl}_{K,S}) \geq 2 + 2 \sqrt{r_{K}+\theta_{K,S}+1},
 \end{equation}
 where $\mathrm{Cl}_{K,S}$ is the ray class group of $K$ modulo $\underset{\mathfrak{q} \in S}{\prod} \mathfrak{q}$ (cf. \cite{HajirMaire2}). The Golod-Shafarevich test~\eqref{eq1} has been used to find many examples of infinite pro-$p$ towers of number fields. However, the test is limited since the failure of the test gives us no information on the infinitude of $G_{S}(K)$. In general, for the test to be successful, either $S$ or the $p$-rank of the class group of $K$ should be large enough. Therefore, when both the $p$-rank of the class group of $K$ and the set $S$ are small, not much work has been done on the infinitude of $G_{S}(K)$. (For an application of Golod-Shafarevich test to $G_{S}(K)$ with small $S$, readers can refer to \cite{HajirMaire}.)

We point out that for proof of the finitude of $G_{S}(K)$, the methods are more limited. As far as we know, there has been no work where the Golod-Shafarevich test was used to prove that some $G_S(K)$ is finite. As a fundamental method we can study the quotients of the lower $p$-central series $\{ \, G_{S}(K)^{(i,p)} \, \}_{i \in \mathbb{N}}$ of $G_{S}(K)$ \cite[Chapter III. \S 8]{NSW}. If we have $G_{S}(K)^{(i,p)}=G_{S}(K)^{(i+1,p)}$ for some $i$, then $G_{S}(K)$ is finite. There is an algorithm of Skopin for computing $G_{S}(K)^{(i,p)}/G_{S}(K)^{(i+1,p)}$ from a presentation of $G_{S}(K)$ \cite{Koch2}, \cite{Skopin}. To effectively use the algorithm, we need enough information on a minimal presentation    
\begin{equation*}
1 \longrightarrow R \longrightarrow \mathcal{F} \longrightarrow G_{S}(K) \longrightarrow 1
\end{equation*}
of $G_{S}(K)$. In fact, there are two difficult problems in understanding the minimal presentations of $G_{S}(K)$. First, it is hard to apply the method in \cite[Chapter 11.4]{Koch} to general number fields. Second, the method in \cite[Chapter 11.4]{Koch} determines elements of $R$ only modulo $[\mathcal{F},\mathcal{F}]^{p}[[\mathcal{F},\mathcal{F}],\mathcal{F}]$ (cf. \cite[Theorem 11.10]{Koch}). This is insufficient for understanding $G_{S}(K)^{(i,p)}/G_{S}(K)^{(i+1,p)}$ for $i \geq 3$. In \cite{BostonLeedhamGreen}, Boston introduced an algorithm for computing the tame pro-$p$ groups, which is obtained by strengthening the $p$-group generation algorithm of O’Brien by using the number theoretic constraints on the Galois groups; the algorithm needs to be implemented on computer programs.

One easy and well-understood general case is when $d(G_{S}(K))=1$. In that case, by Burnside's basis theorem, $G_{S}(K)$ is a pro-$p$ cyclic group. Then by the class field theory, $G_{S}(K)$ is finite (cf. \cite{Maire3}). This can be understood as a specially known case of the general Tame Fontaine-Mazur conjecture since $\mathbb{Z}_{p}$ is $p$-adic analytic. The conjecture states that if $S$ is a finite set of non-$p$-adic places of $K$, then any $p$-adic analytic quotient of $G_{S}(K)$ is finite.  

In this work, we study the finitude of $G_{\{\mathfrak{q}\}}(F)$ for a number field $F$ with cyclic $p$-class group and a non-$p$-adic prime ideal $\mathfrak{q}$ of $F$ by determining whether or not $G_{\{\mathfrak{q}\}}(F)$ is $p$-adic analytic. If $G_{\{\mathfrak{q}\}}(F)$ is $p$-adic analytic, then under the Tame Fontaine-Mazur conjecture, $G_{\{\mathfrak{q}\}}(F)$ is expected to be finite.

Since the late 80s, Lazard's theory of $p$-adic analytic groups \cite{Lazard} has been revisited by focusing more on the uniformly powerful pro-$p$ groups (Definition \ref{uniformlypowerful}) instead of Lazard's saturable groups \cite{DdSMS}. Accordingly, work has been done on the {\it Tame Fontaine-Mazur conjecture-uniform version}, which is equivalent to the Tame Fontaine-Mazur conjecture. The Tame Fontaine-Mazur conjecture-uniform version states that if $S$ is a finite set of non-$p$-adic places of $K$, then any quotient of $G_{S}(K)$ that is uniformly powerful is trivial. 
 This conjecture is already known to be true when $d(G_{S}(K))=1$ or $2$ (cf. the paragraph before Proposition \ref{powerfulTFM} in \S 2). As a consequence, if $d(G_{S}(K)) = 2$ and $G_{S}(K)$ is powerful (Definition \ref{powerful}), then $G_{S}(K)$ is finite unconditionally (Proposition \ref{powerfulTFM}).
Therefore, we study the powerfulness of $G_{\{\mathfrak{q}\}}(F)$. 
 In this work, we focus on the case when $p$ is odd. In particular, we assume that $p$ is odd starting from Proposition \ref{Zelmanovtheorem}; the definitions of the powerfulness of pro-$p$ groups for the case when $p$ is odd and the case when $p=2$ are different.

\vskip 10pt
We briefly describe our main results as follows. If $\mathfrak{q}$ does not split in the $p$-class field tower of $F$, then we obtain the following theorem.

\begin{theorem}\label{maintheorem1}
Let $F$ be a number field whose $p$-class group is non-trivial and cyclic. Let $\mathfrak{q}$ be a non-$p$-adic prime of $F$ which does not split in the $p$-class field tower of $F$. Then $G_{\{\mathfrak{q}\}}(F)$ is powerful and finite.
\end{theorem}

 We briefly describe the idea for the proof of Theorem \ref{maintheorem1} as follows. We first observe that $\mathfrak{q}$ does not split in $F_{\{\mathfrak{q}\}}$. Then $G_{\{\mathfrak{q}\}}(F)$ is isomorphic to a quotient of the Galois group of the maximal pro-$p$ extension of the completion $F_{\mathfrak{q}}$ of $F$ at $\mathfrak{q}$, which is a Demushkin group of generator rank $2$. We use the fact that the Demushkin groups of generator rank $2$ are powerful.

If $\mathfrak{q}$ splits in the $p$-class field tower of $F$, then $G_{\{\mathfrak{q}\}}(F)$ may not be powerful (Proposition \ref{nonpowerful}). However, our next Theorem \ref{maintheorem2} implies that for the majority of such $\mathfrak{q}$'s, $G_{\{\mathfrak{q}\}}(F)$ is powerful and finite. For a technical reason (cf. Proposition \ref{upperbound}), we assume that for $p=3$, $F$ is not a complex quartic number field containing the primitive third root $\zeta_{3}$ of unity.

\begin{theorem}\label{maintheorem2}
Let $p$ be an odd prime. Let $F$ be a number field whose $p$-class group is non-trivial and cyclic. Assume that for $p=3$, $F$ is not a complex quartic number field containing $\zeta_{3}$. Let $r_{F}$ be the $\mathbb{Z}$-rank of the multiplicative group of units of $\mathcal{O}_F$. Let $\mathfrak{M}$ be the set of primes $\mathfrak{q}$ of $F$ which split (not necessarily completely) in the $p$-class field tower of $F$ such that $G_{\{\mathfrak{q}\}}(F)$ has generator rank two. Let $\mathfrak{M}'$ be the subset of $\mathfrak{M}$ consisting of the primes $\mathfrak{q}$ such that $G_{\{\mathfrak{q}\}}(F)$ is infinite. Let $D(\mathfrak{M})$ and $D(\mathfrak{M}')$ be the Dirichlet densities of $\mathfrak{M}$ and $\mathfrak{M}'$, respectively. Then, the ratio $D(\mathfrak{M}')/D(\mathfrak{M})$ is bounded above by $p^{-\mathrm{max}\{r_{F}-1,1\}}$.
\end{theorem}

Throughout this work, $F_{1}$ denotes the unramified cyclic extension of $F$ of degree $p$. When $\mathfrak{q}$ splits in $F_{1}$, let us denote the primes of $F_{1}$ over $\mathfrak{q}$ by $\mathfrak{q}_{1}, \ldots, \mathfrak{q}_{p}$. 

We briefly describe the main idea for the proof of Theorem \ref{maintheorem2} as follows.
To prove Theorem \ref{maintheorem2}, we use group-theoretic arguments to reduce the question of the powerfulness of $G_{\{\mathfrak{q}\}}(F)$ to a class field theoretical question of the non-existence of a degree-$p$ cyclic extension of $F_{1}$, where precisely $\mathfrak{q}_{1},\ldots, \mathfrak{q}_{p-1}$ are ramified (Corollary \ref{criterionramification}). By the Gras-Munnier theorem (Theorem \ref{GrasMunnier}), the latter is answered by studying the Galois group of the governing field $\mathrm{Gov}(F_{1})$ of $F_{1}$ over $F_{1}(\zeta_{p})$ (Proposition \ref{mainproposition}). Using Kummer theory, we can study $V_{\emptyset}(F_{1})/F_{1}^{\times \, p}$ instead. The group $V_{\emptyset}(F_{1})/F_{1}^{\times \, p}$ is approximated by $U_{F_{1}} \otimes_{\mathbb{Z}} \mathbb{F}_{p}$, where $U_{F_{1}}$ denotes the multiplicative group of units of $\mathcal{O}_{F_1}$. Therefore, some knowledge of $\mathbb{F}_{p}[\mathrm{Gal}(F_{1}/F)]$-module structure of $U_{F_{1}} \otimes_{\mathbb{Z}} \mathbb{F}_{p}$ is sufficient to answer the class field theoretic question. The relative Galois module structure of $U_{F_1} \otimes_{\mathbb{Z}} \mathbb{F}_p$ is relatively easy to handle in our work because $\mathrm{Gal}(F_{1}/F)$ is cyclic of prime order, and we have a strong assumption on the $p$-class group of $F$. Our strategy is in the same spirit as the recent applications of the relative Galois module structure of (algebraic) units to various studies of tamely ramified pro-$p$ extensions of number fields. Interested readers can refer to \cite{HajirMaire3, Deficiency, Ozaki, Ozaki2}.

\
 \vskip 10pt
 \noindent \textbf{Notation}
 \vskip 10pt
 For a set $X$, we denote its number of elements by $\sharp X$. The finite field of $p$ elements is denoted by $\mathbb{F}_{p}$. For a group $G$, $\mathbb{F}_{p}$ also denotes the $\mathbb{F}_{p}[G]$-module with the trivial $G$-action. For a vector space $M$ over $\mathbb{F}_{p}$, its dimension is denoted by $\mathrm{dim}_{\mathbb{F}_{p}}M$. For a number field $K$, the class group of $K$ is denoted by $\mathrm{Cl}_{K}$. For a finite abelian group $\mathfrak{A}$, its $p$-primary component is denoted by $\mathfrak{A}(p)$. 

\section{The finitude of powerful tame pro-$p$ groups}

In this section, we explain Proposition \ref{powerfulTFM}, which states that a Galois group of a tamely ramified pro-$p$ extension is finite if it is powerful and its generator rank is two. We recall the notions of powerfulness and uniformly powerfulness of pro-$p$ groups and the Tame Fontaine-Mazur conjecture-uniform version. For more information on the group-theoretic approach to Lazard's theory of $p$-adic analytic groups, readers can refer to \cite{DdSMS}. For a survey on the Tame Fontaine-Mazur conjecture-uniform version, readers can refer to \cite{APRT}.

\begin{definition}\cite{Lazard}
A pro-$p$ group $G$ is $p$-adic analytic if $G$ has the structure of a $p$-adic analytic manifold with the properties
\begin{enumerate}
    \item[(i)] the function $G \times G \rightarrow G$ given by $(x,y) \rightarrow xy$ is analytic;
    \item[(ii)] the function $G \rightarrow G$ given by $x \rightarrow x^{-1}$ is analytic.
\end{enumerate}
\end{definition}

Lazard's theory of $p$-adic analytic groups has been revisited group theoretically. The powerful pro-$p$ groups and uniformly powerful pro-$p$ groups play important roles.

\begin{definition}\label{powerful}
Let $p$ be an odd prime. Let $G$ be a pro-$p$ group. Then $G$ is called \textbf{powerful} if the quotient group $G/\overline{G^{p}}$ is abelian, where the bar denotes the topological closure and $G^{p}$ is the (normal) subgroup of $G$ generated by the $p$th powers of elements of $G$.
\end{definition}

\begin{remark}
When $p$ is even, a pro-$p$ group $G$ is called powerful if $G/\overline{G^{4}}$ is abelian.
\end{remark}

From the definition, the following lemma is immediate.

\begin{lemma}\label{quotient}
Any quotient of a powerful pro-$p$ group is also powerful.
\end{lemma}

\begin{example}\label{Demushkin}
A pro-$p$ group $G$ is a Demushkin group if 
\begin{itemize}
\item[(1)] $\mathrm{dim}_{\mathbb{F}_p}\, H^1(G, \mathbb{F}_p) < \infty$,
\item[(2)] $\mathrm{dim}_{\mathbb{F}_p}\, H^2(G, \mathbb{F}_p) = 1$,
\item[(3)] the cup product $H^1(G, \mathbb{F}_p) \times H^1(G, \mathbb{F}_p) \to H^2(G, \mathbb{F}_p)$ is a non-degenerate bilinear form.
\end{itemize}
From the theorem on minimal presentations of Demushkin pro-$p$ groups (cf. \cite[Theorem 3.9.11]{NSW}, \cite[Theorem 3.9.19]{NSW}), we can verify that Demushkin pro-$p$ groups with generator rank $2$ are powerful. It is worth noting that when $p$ is odd, one can also verify this using \cite[Theorem 5.1.6]{SymondsWeigel}.
\end{example}

We also have the following result.

\begin{proposition}\label{powerfulrank}
Let $G$ be a finitely generated pro-$p$ group. If $G$ is powerful, then for any closed subgroup $H$ of $G$, we have $d(G) \geq d(H)$.
\end{proposition}

\begin{proof}
\cite[Theorem 3.8]{DdSMS}.
\end{proof}

Now let us explain a little more about Tame Fontaine-Mazur conjecture-uniform version.

\begin{definition}\label{uniformlypowerful}
A finitely generated pro-$p$ group $G$ is \textbf{uniformly powerful} if it is powerful and one has $\mathrm{dim}_{\mathbb{F}_{p}}G^{(i,p)}/G^{(i+1,p)}=d(G)$ for all $i$.
\end{definition}

The study of uniformly powerful pro-$p$ groups is helpful in studying general $p$-adic analytic groups through the following theorem.

\begin{theorem}\cite[Theorem 8.32]{DdSMS}\label{LazardGroup}
A topological group $G$ is a $p$-adic analytic group if and only if it has an open subgroup which is a uniformly powerful pro-$p$ group.
\end{theorem}

By Theorem \ref{LazardGroup}, we can prove that Tame Fontaine-Mazur conjecture and Tame Fontaine-Mazur conjecture-uniform version are equivalent. (cf. \cite[\S 5.3]{APRT}). The Tame Fontaine-Mazur conjecture is still out of reach. Even though the uniform version is just equivalent to the original conjecture, the uniform version can be more approachable, because studying the uniformly powerful quotient of $G_{S}(K)$ maximizes the usage of the initial conditions given by $G_{S}(K)$. For the works on the Tame Fontaine-Mazur conjecture-uniform version, readers can refer to \cite{Boston, Boston2, HajirMaire3, Maire, Wingberg}.

Also, the structure of uniformly powerful pro-$p$ groups is known when the generator rank is $1$ or $2$. In that case, a uniformly powerful pro-$p$ group has $\mathbb{Z}_{p}$ as a quotient (cf. \cite[Theorem 3.17]{APRT}, \cite[Exercise 3.11 on p. 60]{DdSMS}). Therefore, by the class field theory, any uniformly powerful quotient $W$ of a tame pro-$p$ group $G_{S}(K)$ with $d(W) \leq 2$ is known to be trivial (cf. \cite[\S 5.4]{APRT}). From this fact, we have the following proposition.

\begin{proposition}\label{powerfulTFM}
Let $K$ be a number field and $S$ a finite set of non-$p$-adic places of $K$. Let $\mathcal{K}$ be a pro-$p$ extension of $K$ which is unramified outside $S$. If the Galois group $H=\mathrm{Gal}(\mathcal{K}/K)$ is powerful with $d(H) \leq 2$, then $H$ is finite.
\end{proposition}

\begin{proof}
Since $H$ is powerful, $H$ has a uniformly powerful open subgroup $H'$ (cf. \cite[Theorem 4.2]{DdSMS}). Let $K'$ be the subfield of $\mathcal{K}$ fixed by $H'$. Then $H'$ is a quotient of $G_{S'}(K')$ where $S'$ is the set of places of $K'$ lying over the places in $S$. By Proposition \ref{powerfulrank}, one has $d(H') \leq d(H) = 2$. Hence, the Tame Fontaine-Mazur conjecture-uniform version is true for $H'$, and $H'$ is trivial. Thus, $H$ is finite.
\end{proof}

\begin{remark}
For an odd prime $p$, the finite powerful $p$-groups with generator rank $2$ were classified in \cite[\S 2]{delasHerasGunnar}.
\end{remark}

\section{Powerfulness of $G_{\{\mathfrak{q}\}}(F)$}

 Let $F$ be a number field with a non-trivial cyclic $p$-class group, and let $\mathfrak{q}$ be a non-$p$-adic prime of $F$. In this section, we study the finitude of $G_{\{\mathfrak{q}\}}(F)$ by studying the powerfulness of $G_{\{\mathfrak{q}\}}(F)$. Since $G_{\{\mathfrak{q}\}}(F)$ is well-known to be finite cyclic if $d(G_{\{\mathfrak{q}\}}(F))=1$, we assume that $d(G_{\{\mathfrak{q}\}}(F))=2$ throughout this work. Therefore, $F$ must have a degree $p$ cyclic extension that is ramified precisely at $\mathfrak{q}$; since the $p$-class group of $F$ is cyclic, there are exactly $p$ extensions of $F$ that satisfy this condition.

First of all, if $\mathfrak{q}$ does not split in the $p$-class field tower of $F$, then we obtain Theorem \ref{maintheorem1} as follows. 

\vskip 10pt

\begin{proof}[\textbf{Proof of Theorem \ref{maintheorem1}}]
Let $M$ be the fixed subfield of $F_{\{\mathfrak{q}\}}$ by the Frattini subgroup of $G_{\{\mathfrak{q}\}}(F)$. Then $\mathrm{Gal}(M/F)$ is isomorphic to $(\mathbb{Z}/p\mathbb{Z})^{2}$, and $\mathfrak{q}$ is ramified in $M$ with the ramification index $p$. Since $\mathfrak{q}$ does not split in the $p$-class field tower, the residue class degree of $\mathfrak{q}$ in $M$ is $p$. Therefore, $\mathrm{Gal}(M/F)$ is equal to the decomposition subgroup at the unique prime of $M$ over $\mathfrak{q}$. By the Burnside basis theorem \cite[Theorem 4.10]{Koch}, for a place $\mathfrak{q}'$ of $F_{\{\mathfrak{q}\}}$ above $\mathfrak{q}$, the decomposition subgroup of $G_{\{\mathfrak{q}\}}(F)$ at $\mathfrak{q}'$ is equal to $G_{\{\mathfrak{q}\}}(F)$. Therefore, $\mathfrak{q}$ does not split in $F_{\{\mathfrak{q}\}}$, and $G_{\{\mathfrak{q}\}}(F)$ is isomorphic to a quotient of the Galois group $\mathrm{Gal}(\overline{F_{\mathfrak{q}}}/F_{\mathfrak{q}})$ of the maximal pro-$p$ extension $\overline{F_{\mathfrak{q}}}$ of the completion $F_{\mathfrak{q}}$ of $F$ at $\mathfrak{q}$. Since $\mathfrak{q}$ is prime to $p$, $\mathrm{Gal}(\overline{F_{\mathfrak{q}}}/F_{\mathfrak{q}})$ is a pro-$p$ Demushkin group of generator rank $2$ \cite[\S 10.1]{Koch}. Since the Demushkin groups of generator rank $2$ are powerful, $G_{\{\mathfrak{q}\}}(F)$ is powerful by Lemma \ref{quotient}. The theorem follows from Proposition \ref{powerfulTFM}.
\end{proof}

Now let us focus on the case when $\mathfrak{q}$ splits in the $p$-class field tower of $F$. From now on, we use the definition of powerfulness in our arguments. Therefore, in the rest of this work, we only consider odd primes $p$. We first have the following group-theoretic observations in Proposition \ref{Zelmanovtheorem} and Lemma \ref{nonabelian2pgroup}.

\begin{proposition}\label{Zelmanovtheorem}
The quotient $G_{\{\mathfrak{q}\}}(F)/\overline{G_{\{\mathfrak{q}\}}(F)^{p}}$ is a finite $p$-group.
\end{proposition}
\begin{proof}
Let
\begin{equation*}
G_{\{\mathfrak{q}\}}(F)/\overline{G_{\{\mathfrak{q}\}}(F)^{p}} \simeq \underset{U}{\varprojlim} \, G_{\{\mathfrak{q}\}}(F)/U
\end{equation*}
be the canonical isomorphism where the inverse limit is over all the open normal subgroups $U$ of $G_{\{\mathfrak{q}\}}(F)$ containing $\overline{G_{\{\mathfrak{q}\}}(F)^{p}}$. Each quotient $G_{\{\mathfrak{q}\}}(F)/U$ of $G_{\{\mathfrak{q}\}}(F)/\overline{G_{\{\mathfrak{q}\}}(F)^{p}}$ has generator rank at most two and exponent $p$. By the theorem of Zel$'$manov on the restricted Burnside problem \cite{Zelmanov}, for two fixed natural numbers $r$ and $s$, the orders of finite groups with  $r$ generators and exponent $s$ are uniformly bounded. (We note that Kostrikin settled the case where $s$ is a prime number in the 1950s \cite{Kostrikin}.) Therefore, there is an open normal subgroup $U'$ of $G_{\{\mathfrak{q}\}}(F)$ containing $\overline{G_{\{\mathfrak{q}\}}(F)}$ such that the order of $G_{\{\mathfrak{q}\}}(F)/U'$ is maximal. By the maximality, $G_{\{\mathfrak{q}\}}(F)/U'$ is equal to $G_{\{\mathfrak{q}\}}(F)/\overline{G_{\{\mathfrak{q}\}}(F)^{p}}$.
\end{proof}

Even though there are only finitely many $p$-groups with generator rank two and exponent $p$, the classification of those groups is still out of reach \cite{Burnsideproblem}. Therefore, it is hard to compute $G_{\{\mathfrak{q}\}}(F)/\overline{G_{\{\mathfrak{q}\}}(F)^{p}}$. Instead, we can use the following fact.

\begin{lemma}\label{nonabelian2pgroup}
Let $G$ be a finite $p$-group with generator rank two and exponent $p$. If $G$ is not abelian, then the quotient $G/[[G,G],G]$ is isomorphic to the Heisenberg group $H_{p}$ of order $p^{3}$.
\end{lemma}
\begin{proof}
Since $G$ is finite, there is a surjective homomorphism $\mathfrak{F} \to G$ from a free group $\mathfrak{F}$ with two generators to $G$. Therefore, there is a surjection $[\mathfrak{F},\mathfrak{F}]/[[\mathfrak{F},\mathfrak{F}],\mathfrak{F}] \to [G,G]/[[G,G],G]$. It is well-known that $[\mathfrak{F},\mathfrak{F}]/[[\mathfrak{F},\mathfrak{F}],\mathfrak{F}]$ is isomorphic to $\mathbb{Z}$ (cf. \cite{Witt}). Hence, we have either $[G,G]=[[G,G],G]$ or $([G,G] : [[G,G],G])=p$. Since a $p$-group is nilpotent, $G$ is abelian if and only if $[G,G]=[[G,G],G]$. In conclusion, $G$ is non-abelian if and only if $G/[[G,G],G]$ has order $p^{3}$. From the classification of $p$-groups of order $p^{3}$ (cf. \cite[\S 5.3.8]{p3}), in that case $G/[[G,G],G]$ is the Heisenberg group of order $p^{3}$.
\end{proof}

By Proposition \ref{Zelmanovtheorem} and Lemma \ref{nonabelian2pgroup}, if $G_{\{\mathfrak{q}\}}(F)$ is not powerful, then $G_{\{\mathfrak{q}\}}(F)$ has $H_{p}$ as its quotient. Let $L$ be a subfield of $F_{\{\mathfrak{q}\}}$ containing $F$ such that one has $\mathrm{Gal}(L/F) \simeq H_{p}$. Let $M$ be the fixed subfield of the Frattini subgroup of $G_{\{\mathfrak{q}\}}(F)$. Then, we have the following field diagram
\begin{equation*}
\begin{tikzcd}
                                                                                                                                 & L                                                                          \\
                                                                                                                                 & M \arrow[u, no head] \arrow[u, "{[H_{p},H_{p}]}"', no head, bend right=60] \\
F_{1} \arrow[ru, no head] \arrow[ruu, "\simeq (\mathbb{Z}/p\mathbb{Z})^{2}" description, no head, bend left=49] &                                                                            \\
F \arrow[u, no head] \arrow[ruu, "\simeq (\mathbb{Z}/p\mathbb{Z})^{2}" description, no head, bend right=49]     &                                                                           
\end{tikzcd}
\end{equation*}

\begin{lemma}\label{lemma-grouptheory}
Let $L$ and $\mathfrak{q}$ be as above. Let us fix a prime of $L$ above $\mathfrak{q}$ and write $D$ and $T$ for the decomposition subgroup and the inertia subgroup of $\mathrm{Gal}(L/F)$ at the prime, respectively. Then the following are true :
\begin{enumerate}
    \item We have $D \subseteq \mathrm{Gal}(L/F_{1})$;
    \item The subgroup $T$ is not normal in $\mathrm{Gal}(L/F)$ and has order $p$.
\end{enumerate}
\end{lemma}

\begin{proof}
The first statement follows because $\mathfrak{q}$ splits in $F_{1}$. Since $\mathfrak{q}$ is prime to $p$, $T$ is cyclic by the class field theory. Since $\mathrm{Gal}(L/F)$ has exponent $p$, $T$ cannot have order larger than $p$. Hence, we have $\sharp T=p$. If $T$ is a normal subgroup of $\mathrm{Gal}(L/F)$, then its fixed subfield is an unramified abelian extension of $F$ with degree $p^{2}$. We have a contradiction by the assumption on the $p$-class field tower of $F$ and the exponent of $\mathrm{Gal}(L/F)$.
\end{proof}

We can obtain the following necessary conditions for the powerfulness of $G_{\{\mathfrak{q}\}}(F)$ from Lemma \ref{lemma-grouptheory}.

\begin{proposition}\label{G_S/G_S^{p}}
Suppose that $G_{\{ \mathfrak{q}\}}(F)/\overline{G_{\{\mathfrak{q}\}}(F)^{p}}$ is not abelian. Let $L$ and $M$ be the subfields of $F_{\{\mathfrak{q}\}}$ as above. Then the following are true :
\begin{enumerate}
    \item The extension $L/M$ is unramified;
    \item Let $\mathfrak{q}_{1},\ldots, \mathfrak{q}_{p}$ be the primes of $F_{1}$ above $\mathfrak{q}$. Then for each $1 \leq i \leq p$, $F_{1}$ has a degree $p$ cyclic extension that is precisely ramified at the set $\{ \mathfrak{q}_j \, | \, 1 \leq j \leq p , j \neq i \}$
\end{enumerate}
\end{proposition}

\begin{proof}
\begin{enumerate}
    \item Let $T$ be the subgroup of $\mathrm{Gal}(L/F)$ in Lemma \ref{lemma-grouptheory}. The subgroups of $\mathrm{Gal}(L/F)$ conjugate to $T$ are the inertia subgroups at the primes of $L$ above $\mathfrak{q}_{1},\ldots, \mathfrak{q}_{p}$. They are $p$ cyclic subgroups of $\mathrm{Gal}(L/F_{1}) \simeq (\mathbb{Z}/p\mathbb{Z})^{2}$. The remaining cyclic subgroup of $\mathrm{Gal}(L/F_{1})$ is normal in $\mathrm{Gal}(L/F)$, and it is equal to the commutator subgroup $[\mathrm{Gal}(L/F), \mathrm{Gal}(L/F)]$ whose fixed field is $M$. Since $L/F$ is unramified outside $\mathfrak{q}$ and $\mathrm{Gal}(L/M)$ intersects trivially with all the conjugates of $T$, $L/M$ is an unramified extension.
   \item For each $\mathfrak{q}_{i}$ with $1 \leq i \leq p$, the inertia subgroup $T_{i}$ of $\mathrm{Gal}(L/F_{1})$ at $\mathfrak{q}_{i}$ is conjugate to $T$ in $\mathrm{Gal}(L/F)$. Since the conjugates of $T$ are all distinct, the fixed subfield of $T_{i}$ is ramified over $F_{1}$ precisely at the set $\{ \mathfrak{q}_j \, | \, 1 \leq j \leq p , j \neq i \}$
\end{enumerate}
\end{proof}

\begin{corollary}\label{pclassnumber}
Let $M$ be as above. Then, $G_{\{\mathfrak{q}\}}(F)$ is powerful if the $p$-class number of $M$ is $1$.
\end{corollary}

It is difficult to use Corollary \ref{pclassnumber} in practice. Instead, in the next section, we use the following necessary condition.

\begin{corollary}\label{criterionramification}
The Galois group $G_{\{ \mathfrak{q}\}}(F)$ is powerful if for some $1 \leq i \leq p$, $F_{1}$ does not admit a degree $p$ cyclic extension which is ramified precisely at the set $\{ \mathfrak{q}_j \, | \, 1 \leq j \leq p , j \neq i \}$.
\end{corollary}

\section{An application of the Gras-Munnier theorem to the case when $\mathfrak{q}$ splits in $F_{1}$}

In this section, we study the powerfulness of $G_{\{\mathfrak{q}\}}(F)$ for $\mathfrak{q}$ which splits in $F_{1}$ by using Corollary \ref{criterionramification}. We effectively apply the theorem of Gras and Munnier.

\subsection{Gras-Munnier theorem}

Let $K$ be a number field, and let $S$ be a finite set of finite non-$p$-adic primes of $K$. The theorem of Gras and Munnier gives us a criterion for the existence of a cyclic extension of $K$ of degree $p$ which is ramified precisely at $S$. We assume that for each $v \in S$, its ideal norm is congruent to $1$ modulo $p$, because otherwise $v$ cannot be ramified in a pro-$p$ extension of $K$. For $K$ and $S$, let us define the multiplicative subgroup $V_{S}(K)$ of $K^{\times}$ by 
\begin{equation*}
    V_{S}(K) = \big \{ \,\, x \in K^{\times} \,\, | \,\,\, (x) = \mathfrak{a}^{p} \,\,\text{for a fractional ideal $\mathfrak{a}$ of $K$} \,\,\, \& \,\,\, x \in K_{v}^{\times \, p} \,\text{for all $v \in S$} \,\,\, \big \},
\end{equation*}
where $K_{v}$ denotes the completion of $K$ at the place $v$ (cf. \cite[\S 11.2]{Koch}).
\begin{definition}\label{governing}
Let $K$ be a number field. Let $\emptyset$ be the empty set of primes of $K$. The \textbf{governing field} of $K$ is the field $K(\zeta_{p}, \sqrt[p]{V_{\emptyset}(K)})$ obtained from $K$ by adjoining the primitive $p$th roots of unity and $p$th roots of elements of $V_{\emptyset}(K)$. We denote this field by $\mathrm{Gov}(K)$.
\end{definition}

Let us choose a prime $v'$ of $K(\zeta_{p})$ above $v$ for each prime $v \in S$. Since $\mathrm{Gal}(\mathrm{Gov}(K)/K(\zeta_{p}))$ is finite and abelian with exponent $p$, $\mathrm{Gal}(\mathrm{Gov}(K)/K(\zeta_{p}))$ is a finite dimensional vector space over $\mathbb{F}_{p}$. In \cite{GrasMunnier}, the authors proved the following theorem. We recommend \cite[Chapter V]{Gras} for a more comprehensive reference and \cite{HajirMaireRamakrishna} for a short proof of the theorem.

\begin{theorem}(Gras-Munnier)\label{GrasMunnier}
Let $S$ be a finite set of finite non-$p$-adic primes of $K$ whose ideal norms are congruent to $1$ modulo $p$. Then there is a cyclic extension of degree $p$ over $K$ which is ramified precisely at $S$ if and only if there is a relation
\begin{equation*}
   \underset{v \in S}{\prod}  \bigg ( \,\, \frac{\mathrm{Gov}(K)/K(\zeta_{p})}{v'} \,\, \bigg )^{a_{v}} = 1 \in \mathrm{Gal}(\mathrm{Gov}(K)/K(\zeta_{p}))
\end{equation*}
among the Frobenius automorphisms of $\mathrm{Gal}(\mathrm{Gov}(K)/K(\zeta_{p}))$ at $v'$ such that $a_{v} \in \mathbb{F}_{p}^{\times}$ for all $v \in S$.
\end{theorem}

\begin{remark}
From the definition of $V_{\emptyset}(K)$, $\mathrm{Gov}(K)/K(\zeta_{p})$ is unramified at all the non-$p$-adic primes of $K(\zeta_{p})$. (cf. \cite[Exercise 9.1]{Washington}). A different choice of $v'$ changes the Frobenius automorphism by a power prime to $p$; therefore, the choice is not important to the theorem.
\end{remark}

The powerfulness of $G_{\{\mathfrak{q}\}}(F)$ is a very rigid condition. If $G_{\{\mathfrak{q}\}}(F)$ is powerful, then all the subgroups of $G_{\{\mathfrak{q}\}}(F)$ have generator rank at most $d(G_{\{\mathfrak{q}\}}(F)) \leq 2$ by Proposition \ref{powerfulrank}. 

\begin{proposition}\label{nonpowerful}
Let $p$ be an odd prime. Then there are infinitely many $\mathfrak{q}$ such that $G_{\{\mathfrak{q}\}}(F)$ is not powerful.
\end{proposition}

\begin{proof}
Suppose that $\mathfrak{q}$ splits in $F_1$ and let $\mathfrak{q}_{1}, \ldots, \mathfrak{q}_{p}$ be the primes of $F_{1}$ above $\mathfrak{q}$. The Galois group $G_{\{\mathfrak{q}\}}(F)$ is not powerful if the ray class group of $F_{1}$ modulo $\mathfrak{q}\mathcal{O}_{F_1}=\mathfrak{q}_{1} \cdots \mathfrak{q}_{p}$ has $p$-rank larger than $2$. By the Chebotarev density theorem, there are infinitely many primes $\mathfrak{q}$ of $F$ which split completely in $\mathrm{Gov}(F_{1})$. Then by Theorem \ref{GrasMunnier}, for each $1 \leq i \leq p$, there is a cyclic $p$-extension $L_{i}$ of $F_{1}$ which is ramified precisely at $\mathfrak{q}_{i}$. By considering the ramification, we can check that $\mathrm{Gal}(L_{1} \cdots L_{p}/F_{1})$ is isomorphic to $(\mathbb{Z}/p\mathbb{Z})^{p}$. Therefore, the ray class group of $F_{1}$ modulo $\underset{i=1}{\overset{p}{\prod}}\mathfrak{q}_{i}$ has $p$-rank at least $p$.
\end{proof}

\begin{remark}
Proposition \ref{nonpowerful} is in the same spirit as the strategy of \cite{HajirMaire} to prove the infinitude of $G_{S}(K)$ for small $S$. In \cite{HajirMaire}, for certain $K$ and $S$, the authors used the Gras-Munnier theorem to prove the existence of a subgroup $H$ of $G_{S}(K)$ with a large generator rank. Therefore, even though the Golod-Shafarevich test fails for $G_{S}(K)$, it can work for $H$, which leads to the conclusion that $G_{S}(K)$ is infinite.
\end{remark}

\subsection{Proof of Theorem \ref{maintheorem2}}\label{proofofTheorem1.2}

In this section, we prove Theorem \ref{maintheorem2} by studying the Galois group $\mathrm{Gal}(\mathrm{Gov}(F_{1})/F_{1}(\zeta_{p}))$. 

\subsubsection{The $\mathbb{F}_{p}[\mathrm{Gal}(F_{1}/F)]$-module structure of $\mathrm{Gal}(\mathrm{Gov}(F_{1})/F_{1}(\zeta_{p}))$}\label{subsubsection4.2.1}

According to Theorem \ref{maintheorem1} and Theorem \ref{GrasMunnier}, unless $\mathfrak{q}$ splits completely in $\mathrm{Gov}(F)F_{1}$, $G_{\{\mathfrak{q}\}}(F)$ is already known to be finite. Therefore, we focus on the non-$p$-adic primes $\mathfrak{q}$ which split completely in $F_{1}(\zeta_{p})$ such that the Frobenius automorphisms of $\mathrm{Gal}(\mathrm{Gov}(F_{1})/F_{1}(\zeta_{p}))$ at the primes of $F_{1}(\zeta_{p})$ over $\mathfrak{q}$ fix $\mathrm{Gov}(F)$.

As $F_{1}/F$ has degree $p$, $\mathrm{Gal}(F_{1}/F)$ is isomorphic to $\mathrm{Gal}(F_{1}(\zeta_{p})/F(\zeta_{p}))$. Since $V_{\emptyset}(F_{1})$ is invariant under the action of $\mathrm{Gal}(F_{1}/F)$, $\mathrm{Gov}(F_{1})$ is Galois over $F(\zeta_{p})$. Therefore, under the identification $\mathrm{Gal}(F_{1}/F) \simeq \mathrm{Gal}(F_{1}(\zeta_{p})/F(\zeta_{p}))$, $\mathrm{Gal}(F_{1}/F)$ acts on $\mathrm{Gal}(\mathrm{Gov}(F_{1})/F_{1}(\zeta_{p}))$ by inner automorphisms.

Since $\mathrm{Gal}(F_{1}/F)$ acts transitively on the primes $\mathfrak{q}_{1},\ldots, \mathfrak{q}_{p}$ of $F_{1}$ over $\mathfrak{q}$, it is natural to study the $\mathbb{F}_{p}[\mathrm{Gal}(F_{1}/F)]$-module structure of $\mathrm{Gal}(\mathrm{Gov}(F_{1})/F_{1}(\zeta_{p}))$. Let $\sigma$ be a fixed generator of $\mathrm{Gal}(F_{1}/F)$. Then $\mathbb{F}_{p}[\mathrm{Gal}(F_{1}/F)]$ and $\mathbb{F}_{p}[X]/(X^{p}-1)$ are isomorphic as rings by the homomorphism sending $\sigma$ to the class of $X$. Hence, $\mathrm{Gal}(\mathrm{Gov}(F_{1})/F_{1}(\zeta_{p}))$ is also a module over $\mathbb{F}_{p}[X]$ via the projection map $\mathbb{F}_{p}[X] \to \mathbb{F}_{p}[X]/(X^{p}-1)$.

\begin{proposition}\label{mainproposition}
Let $\mathfrak{q}$ be a non-$p$-adic prime of $F$ which splits completely in $F_{1}(\zeta_{p})$. Let $\mathfrak{P}$ be a prime of $F_{1}(\zeta_{p})$ over $\mathfrak{q}$. Let $\tau \in \mathrm{Gal}(\mathrm{Gov}(F_{1})/F_{1}(\zeta_{p}))$ be the Frobenius automorphism at $\mathfrak{P}$. If $\tau$ fixes $\mathrm{Gov}(F)$ and is not annihilated by $\Psi(X) := (X-1)^{p-2} \in \mathbb{F}_{p}[X]$, then $G_{\{\mathfrak{q}\}}(F)$ is powerful with generator rank two.
\end{proposition}

\begin{proof}

 Since $\mathfrak{q}$ splits completely in $F_{1}(\zeta_{p})$, the norm $\mathrm{N}_{F_{1}(\zeta_{p})/F(\zeta_{p})}\mathfrak{P}$ is equal to a prime ideal of $F(\zeta_{p})$ above $\mathfrak{q}$. Therefore, we have
\begin{equation*}
    \tau \, \big|_{\mathrm{Gov}(F)} = \bigg ( \frac{\mathrm{Gov}(F_{1})/F_{1}(\zeta_{p})}{\mathfrak{P}} \bigg) \bigg|_{\mathrm{Gov}(F)} = \bigg ( \frac{\mathrm{Gov}(F)/F(\zeta_{p})}{\mathrm{N}_{F_{1}(\zeta_{p})/F(\zeta_{p})}\mathfrak{P}} \bigg ) = 1.
\end{equation*}
Hence, we have $d(G_{\{\mathfrak{q}\}}(F))=2$ by Theorem \ref{GrasMunnier}.
     Let $\mathfrak{P}_{2}, \ldots, \mathfrak{P}_{p}$ be the conjugates of $\mathfrak{P}=\mathfrak{P}_{1}$ over $F(\zeta_{p})$. They are distinct because $\mathfrak{q}$ splits completely in $F_{1}(\zeta_{p})$. Then the primes of $F_{1}$ lying below $\mathfrak{P}_{1}, \ldots, \mathfrak{P}_{p}$ are the $p$ primes $\mathfrak{q}_{1},\ldots, \mathfrak{q}_{p}$ above $\mathfrak{q}$. By rearranging the index if necessary, we can assume that $\mathfrak{P}_{i} = \sigma^{i-1}\mathfrak{P}$ for $1 \leq i \leq p$, where we have used the notation $\sigma$ also for the $F(\zeta_{p})$-linear extension of $\sigma \in \mathrm{Gal}(F_{1}/F)$ to $F_{1}(\zeta_{p})$. By Proposition \ref{G_S/G_S^{p}}, if $G_{\{\mathfrak{q}\}}(F)$ is not powerful, then $F_{1}$ admits a cyclic extension of degree $p$ which is ramified precisely at the primes $\mathfrak{q}_{1},\ldots, \mathfrak{q}_{p-1}$. In that case, by Theorem \ref{GrasMunnier}, there are $a_{0}, a_{1}, \ldots, a_{p-2} \in \mathbb{F}_{p}^{\times}$ such that
\begin{equation*}
    \underset{i=1}{\overset{p-1}{\prod}} \, \bigg ( \frac{\mathrm{Gov}(F_{1})/F_{1}(\zeta_{p})}{\mathfrak{P}_{i}} \bigg )^{a_{i-1}} = \underset{i=1}{\overset{p-1}{\prod}} \, \bigg ( \frac{\mathrm{Gov}(F_{1})/F_{1}(\zeta_{p})}{\sigma^{i-1}\mathfrak{P}} \bigg )^{a_{i-1}} = \bigg ( \frac{\mathrm{Gov}(F_{1})/F_{1}(\zeta_{p})}{\mathfrak{P}} \bigg )^{\underset{i=0}{\overset{p-2}{\sum}}a_{i}\sigma^{i}} = 1
\end{equation*}
Therefore, $\tau$ is annihilated by $\Pi(X) := \underset{i=0}{\overset{p-2}{\sum}}a_{i}X^{i} \in \mathbb{F}_{p}[X]$. Since $\tau$ is annihilated by $X^{p}-1=(X-1)^{p}$, the annihilator of $\tau$ in $\mathbb{F}_p[X]$ must be $((X-1)^{m})$ for some $m \leq p-2$. Hence, we have a contradiction; thus $G_{\{\mathfrak{q}\}}(F)$ must be powerful.

\end{proof}

For an $\mathbb{F}_{p}[X]$-module $M$, let $M[\Psi]$ be the kernel of $\Psi(X)$ on $M$. By Proposition \ref{mainproposition}, we can prove Theorem \ref{maintheorem2} by studying the ratio
\begin{equation}\label{ratio}
\frac{\,\,\,\,\, \sharp \, \mathrm{Gal}(\mathrm{Gov}(F_{1})/\mathrm{Gov}(F)F_{1})[\Psi] \,\,\,\,\,}{[\mathrm{Gov}(F_{1}) : \mathrm{Gov}(F)F_{1}]}.
\end{equation}

\vskip 15pt

\subsubsection{$\mathbb{F}_{p}[\mathrm{Gal}(F_{1}/F)]$-module structure of $V_{\emptyset}(F_{1})/F_{1}^{\times \, p}$}

\vskip 10pt

From the definition of the governing field $\mathrm{Gov}(F_{1})$, we have the following non-degenerate Kummer pairing
\begin{equation*}
    \mathrm{Gal}(\mathrm{Gov}(F_{1})/F_{1}(\zeta_{p})) \times \frac{V_{\emptyset}(F_{1})F_{1}(\zeta_{p})^{\times \, p}}{F_{1}(\zeta_{p})^{\times \, p}} \longrightarrow \mu_{p},
\end{equation*}
where $\mu_{p}$ denotes the group of $p$th roots of unity.
For a general number field $L$, it is not difficult to check that $L^{\times}/L^{\times \, p} \to L(\zeta_{p})^{\times}/L(\zeta_{p})^{\times \, p}$ is injective. Hence, $V_{\emptyset}(F_{1})F_{1}(\zeta_{p})^{\times \, p}/F_{1}(\zeta_{p})^{\times \, p}$ is isomorphic to $V_{\emptyset}(F_{1})/F_{1}^{\times \, p}$, and thus, we have a pairing
\begin{equation*}
\mathrm{Gal}(\mathrm{Gov}(F_{1})/F_{1}(\zeta_{p})) \times V_{\emptyset}(F_{1})/F_{1}^{\times \, p} \longrightarrow \mu_{p},
\end{equation*}
which is equivariant over $\mathrm{Gal}(F_{1}/F) \simeq \mathrm{Gal}(F_{1}(\zeta_{p})/F(\zeta_{p}))$ (cf. \cite[I. \S 6]{Gras}). Since $\mathrm{Gal}(F_{1}(\zeta_{p})/F(\zeta_{p}))$ fixes $\mu_{p}$, $\mu_{p}$ is isomorphic to $\mathbb{F}_{p}$ as $\mathbb{F}_{p}[\mathrm{Gal}(F_{1}/F)]$-modules (cf. \cite[Lemma 4.13]{Koch}). Therefore, we have an isomorphism
\begin{equation*}
\mathrm{Gal}(\mathrm{Gov}(F_{1})/F_{1}(\zeta_{p})) \simeq \mathrm{Hom}_{\mathbb{F}_{p}} \big ( \, V_{\emptyset}(F_{1})/F_{1}^{\times \, p} , \, \mathbb{F}_{p} \, \big )
\end{equation*}
of $\mathbb{F}_{p}[\mathrm{Gal}(F_{1}/F)]$-modules, where $\mathrm{Hom}_{\mathbb{F}_{p}}(V_{\emptyset}(F_{1})/F_{1}^{\times \, p}, \mathbb{F}_{p})$ is equipped with the contragradient $\mathrm{Gal}(F_{1}/F)$-action.
By the Kummer theory, we also have an isomorphism
\begin{equation}\label{torsionmodule}
\mathrm{Gal}(\mathrm{Gov}(F_{1})/\mathrm{Gov}(F)F_{1}) \simeq \mathrm{Hom}_{\mathbb{F}_{p}} \big ( \, V_{\emptyset}(F_{1})/V_{\emptyset}(F)F_{1}^{\times \, p}, \mathbb{F}_{p} \big )
\end{equation}
as $\mathbb{F}_{p}[\mathrm{Gal}(F_{1}/F)]$-modules.
Hence, we can study the ratio $(\ref{ratio})$ by analyzing $\mathbb{F}_{p}[\mathrm{Gal}(F_{1}/F)]$-module structure of $V_{\emptyset}(F_{1})/V_{\emptyset}(F)F_{1}^{\times \, p}$. In particular, we have the following lemma.

\begin{lemma}\label{dualmap}
Let $M$ be a module over $\mathbb{F}_{p}[\mathrm{Gal}(F_{1}/F)] \simeq \mathbb{F}_{p}[X]/(X^{p}-1)$. Let $N$ be the dual $\mathrm{Hom}_{\mathbb{F}_{p}}(M,\mathbb{F}_{p})$ of $M$. If $M$ is finite dimensional over $\mathbb{F}_{p}$, then we have $\mathrm{dim}_{\mathbb{F}_{p}} M[\Psi] = \mathrm{dim}_{\mathbb{F}_{p}} N[\Psi]$.
\end{lemma}
\begin{proof}
Let $0 \rightarrow M[\Psi] \rightarrow M \xrightarrow{\Psi(X)} M \rightarrow M' \rightarrow 0$ be the tautological exact sequence. Since the functor $P \to \mathrm{Hom}_{\mathbb{F}_{p}}(P,\mathbb{F}_{p})$ on the category of $\mathbb{F}_p$-vector spaces is exact, we have the exact sequence
\begin{equation*}
    0 \longrightarrow \mathrm{Hom}_{\mathbb{F}_{p}}(M',\mathbb{F}_{p}) \longrightarrow N \xrightarrow{\,\, \Psi^{\ast} \,\,} N \longrightarrow \mathrm{Hom}_{\mathbb{F}_{p}}(M[\Psi], \mathbb{F}_{p}) \longrightarrow 0,
\end{equation*}
where $\Psi^{\ast}$ denotes the dual map of $\Psi(X)$ on $N$. From the definition of the contragradient $\mathrm{Gal}(F_{1}/F)$-action on $\mathrm{Hom}_{\mathbb{F}_{p}}(M,\mathbb{F}_{p})$, the dual map of $(x-1)^{p-2}=(\sigma-1)^{p-2}$ is equal to $(\sigma^{-1}-1)^{p-2}=(x^{-1}-1)^{p-2}=(-x)^{-p+2}(x-1)^{p-2}$. Since $(-x)^{-p+2}$ is invertible, the kernel of $\Psi^{\ast}$ on $N$ is equal to $N[\Psi]$. If $\mathrm{dim}_{\mathbb{F}_{p}}M$ is finite, then we have $\mathrm{dim}_{\mathbb{F}_{p}} M[\Psi] = \mathrm{dim}_{\mathbb{F}_{p}}M' = \mathrm{dim}_{\mathbb{F}_{p}}\mathrm{Hom}_{\mathbb{F}_{p}}(M',\mathbb{F}_{p}) = \mathrm{dim}_{\mathbb{F}_{p}}N[\Psi]$.
\end{proof}

By Lemma \ref{dualmap} and the Kummer duality $(\ref{torsionmodule})$, the ratio $(\ref{ratio})$ is equal to 
\begin{equation}\label{ratio2}
\frac{\,\,\,\,\, \sharp \, \big ( V_{\emptyset}(F_{1})/V_{\emptyset}(F)F_{1}^{\times \, p} \big )[\Psi] \,\,\,\,\,}{\sharp \, V_{\emptyset}(F_{1})/V_{\emptyset}(F)F_{1}^{\times \, p}}.
\end{equation}

Since $\mathbb{F}_{p}[X]/(X^{p}-1)$ is a principal ideal domain and $X^{p}-1=(X-1)^{p} \in \mathbb{F}_{p}[X]$, any finitely generated $\mathbb{F}_{p}[X]/(X^{p}-1)$-module $M$ is isomorphic to
\begin{equation*}
\underset{i=1}{\overset{p}{\bigoplus}}\, \big ( \mathbb{F}_{p}[X]/((X-1)^{i}) \big )^{a_{i}}
\end{equation*}
as $\mathbb{F}_{p}[X]/(X^{p}-1)$-modules for unique non-negative integers $a_{i}$ (cf. Jordan normal form). For convenience, let us denote $\mathbb{F}_{p}[X]/((X-1)^{i})$ by $Y_{i}$ for each $1 \leq i \leq p$. We note that $Y_{1}$ is isomorphic to $\mathbb{F}_{p}$ and $Y_{p}$ is the group ring $\mathbb{F}_{p}[\mathrm{Gal}(F_{1}/F)]$ over $\mathbb{F}_{p}$, under the identification $\mathbb{F}_{p}[X]/(X^{p}-1) \simeq \mathbb{F}_{p}[\mathrm{Gal}(F_{1}/F)]$.

For a number field $L$, let us denote the group of units of $\mathcal{O}_L$ by $U_{L}$. Let $\mu_{L}$ be the subgroup of $U_{L}$ of the roots of unity in $L$. Define $\delta_{L} : = 1$ if $\mu_{p} \subset L^{\times}$ and $\delta_{L}:=0$ otherwise.

The module $V_{\emptyset}(F_{1})/F_{1}^{\times \, p}$ can be analyzed by using the well-known exact sequence (cf. \cite[\S 11.2]{Koch})
\begin{equation}\label{Shafarevichsequence}
    0 \longrightarrow U_{F_{1}} \otimes_{\mathbb{Z}} \mathbb{F}_{p} \longrightarrow V_{\emptyset}(F_{1})/F_{1}^{\times \, p} \longrightarrow \mathrm{Cl}_{F_{1}}[p] \longrightarrow 0,
\end{equation}
where $\mathrm{Cl}_{F_{1}}[p]$ denotes the subgroup of $\mathrm{Cl}_{F_{1}}$ of elements of order $p$. The strong assumption on the cyclicity of the $p$-class group of $F$ gives us the following information on $U_{F_{1}} \otimes_{\mathbb{Z}} \mathbb{F}_{p}$.

\begin{proposition}\label{Galoismoduleunit}
We have the following information on the $\mathbb{F}_{p}[X]/(X^{p}-1)$-module structure of $U_{F_{1}} \otimes_{\mathbb{Z}} \mathbb{F}_{p}$ :
\begin{enumerate}
\item If $\delta_{F_{1}}=0$, then $U_{F_{1}} \otimes_{\mathbb{Z}} \mathbb{F}_{p}$ is isomorphic to $Y_{p}^{r_{F}} \oplus Y_{p-1}$;
\item If $\delta_{F_{1}}=1$ and $\mu_{F_{1}}(p) \neq \mu_{F}(p)$, then $U_{F_{1}} \otimes_{\mathbb{Z}} \mathbb{F}_{p}$ is isomorphic to $Y_{p}^{r_{F}} \oplus Y_{p-1} \oplus \mathbb{F}_{p}$;
\item If $\delta_{F_{1}} = 1$ and $\mu_{F_{1}}(p) = \mu_{F}(p)$, then $U_{F_{1}} \otimes_{\mathbb{Z}} \mathbb{F}_{p}$ is isomorphic to $Y^{r_F}_p \oplus Y_{p-1} \oplus \mathbb{F}_p$ or $Y^{r_F-1}_p \oplus Y^2_{p-1} \oplus \mathbb{F}^2_p$.
\end{enumerate}
Therefore, the multiplicity of $Y_{p}$ in the Krull-Schmidt decomposition of $U_{F_{1}} \otimes_{\mathbb{Z}} \mathbb{F}_{p}$ is at least $r_{F}-1$.
\end{proposition}

\begin{proof}
Let $E_{F_{1}}$ be the quotient of $U_{F_{1}}$ by $\mu_{F_{1}}$. Since $E_{F_{1}}$ is torsion free over $\mathbb{Z}$, we have the exact sequence
\begin{equation}\label{modpsequence}
0 \longrightarrow \mu_{F_{1}} \otimes_{\mathbb{Z}} \mathbb{F}_{p} \longrightarrow U_{F_{1}} \otimes_{\mathbb{Z}} \mathbb{F}_{p} \longrightarrow E_{F_{1}} \otimes_{\mathbb{Z}} \mathbb{F}_{p} \longrightarrow 0.
\end{equation}

By the Krull-Schmidt theorem and the theorem of Diederichsen (cf. \cite[\S 74]{CurtisReiner}, \cite{Diederichsen}, \cite[\S 2]{HellerReiner}), we have an isomorphism
\begin{equation}\label{Zpstructure}
E_{F_{1}} \otimes_{\mathbb{Z}} \mathbb{Z}_{p} \simeq \mathbb{Z}_{p}^{a} \oplus \big ( \mathbb{Z}_{p}[X]/(1+X+\cdots + X^{p-1}) \big )^{b} \oplus \big ( \mathbb{Z}_{p}[X]/(X^{p}-1) \big )^{c}
\end{equation}
as $\mathbb{Z}_{p}[X]/(X^{p}-1)$-modules for unique non-negative integers $a,b,c$. (We are using the identification $\mathbb{Z}_{p}[\mathrm{Gal}(F_{1}/F)] \simeq \mathbb{Z}_{p}[X]/(X^{p}-1)$.) By applying the tensor product $\otimes_{\mathbb{Z}_{p}} \mathbb{F}_{p}$ to $(\ref{Zpstructure})$, we obtain
\begin{equation*}
E_{F_{1}} \otimes_{\mathbb{Z}} \mathbb{F}_{p} \simeq \mathbb{F}^a_p \oplus Y^b_{p-1} \oplus Y^c_p. 
\end{equation*}

By the Dirichlet-Herbrand theorem (cf. \cite[Lemma I.3.6]{Gras}), we also have the isomorphism
\begin{equation}\label{rational representation}
E_{F_1} \otimes_{\mathbb{Z}} \mathbb{Q}_p \simeq \mathbb{Q}_p[X]/(1+X+ \cdots + X^{p-1}) \oplus \big ( \mathbb{Q}_p[X]/(X^p-1) \big )^{r_{F}}
\end{equation}
of $\mathbb{Q}_p[X]/(X^p-1)$-modules. By the Chinese remainder theorem, we have the isomorphism
\begin{equation*}
    \mathbb{Q}_p[X]/(X^p-1) \simeq \mathbb{Q}_p \oplus \mathbb{Q}_p[X]/(1+X+\cdots +X^{p-1})
\end{equation*}
of $\mathbb{Q}_p[X]/(X^p-1)$-modules. Comparing the right hand expression of $(\ref{rational representation})$ and the $\mathbb{Q}_p$-tensor of the right hand expression of $(\ref{Zpstructure})$, we obtain $a + 1= b$ and $b+c = r_F + 1$. The lattices $\mathbb{Z}_p$ and $\mathbb{Z}_p[X]/(X^p-1)$ have trivial first cohomology, and one has the isomorphism $H^1(\mathrm{Gal}(F_1/F), \mathbb{Z}_p[X]/(1+X+ \cdots + X^{p-1})) \simeq \mathbb{F}_p$. Hence, the number $b$
is uniquely determined by computing $H^{1}(\mathrm{Gal}(F_{1}/F), E_{F_{1}} \otimes_{\mathbb{Z}} \mathbb{Z}_{p})$; as a result, we get
\begin{equation*}
a=b-1 \,\, , \,\, b=\mathrm{rk}_{p}H^{1}(\mathrm{Gal}(F_{1}/F),E_{F_{1}} \otimes_{\mathbb{Z}} \mathbb{Z}_{p}) \,\, , \,\, c = (r_{F}+1)-b.
\end{equation*}
We can study $H^{1}(\mathrm{Gal}(F_{1}/F), E_{F_{1}} \otimes_{\mathbb{Z}} \mathbb{Z}_{p})$ by computing $H^{1}(\mathrm{Gal}(F_{1}/F), U_{F_{1}} \otimes_{\mathbb{Z}} \mathbb{Z}_{p}) \simeq$ $ H^{1}(\mathrm{Gal}(F_{1}/F), U_{F_{1}})$ and using the long exact sequence of cohomology groups associated to the exact sequence
\begin{equation*}
0 \longrightarrow \mu_{F_{1}}(p) \longrightarrow U_{F_{1}} \otimes_{\mathbb{Z}} \mathbb{Z}_{p} \longrightarrow E_{F_{1}} \otimes_{\mathbb{Z}} \mathbb{Z}_{p} \longrightarrow 0.
\end{equation*}

By a theorem of Iwasawa \cite{Iwasawa}, $H^{1}(\mathrm{Gal}(F_{1}/F), U_{F_{1}})$ is isomorphic to the kernel of the canonical map $I_{F_{1}}^{\mathrm{Gal}(F_{1}/F)}/P_{F} \rightarrow \mathrm{Cl}_{F_{1}}$, where $I_{F_{1}}$ (resp. $P_{F}$) is the group of fractional ideals (resp. principal fractional ideals) of $F_{1}$ (resp. $F$). Since $F_{1}/F$ is unramified, $I_{F_{1}}^{\mathrm{Gal}(F_{1}/F)}/P_{F}$ is isomorphic to $\mathrm{Cl}_{F}$, and the $p$-group $H^{1}(\mathrm{Gal}(F_{1}/F), U_{F_{1}})$ is isomorphic to the kernel of the transfer map $\mathrm{Cl}_{F}(p) \rightarrow \mathrm{Cl}_{F_{1}}(p)$. Since $F_{1}$ and $F$ share the same $p$-Hilbert class field, the capitulation kernel is equal to $\mathrm{Cl}_{F}[p] \simeq \mathbb{F}_p$. Hence, we have $H^{1}(\mathrm{Gal}(F_{1}/F),U_{F_{1}}) \simeq \mathbb{F}_p$.

If $\delta_{F}=0$ or $\mu_{F}(p) \neq \mu_{F_{1}}(p)$, then we have $\hat{H}^{i}(\mathrm{Gal}(F_{1}/F), \mu_{F_{1}}(p)) = 0$ for all $i \in \mathbb{Z}$ (cf. \cite[Lemma 5.4.4(1)]{Popescu}). Therefore, we can conclude that $b$ is equal to $1$ and there exists an isomorphism $E_{F_{1}} \otimes_{\mathbb{Z}} \mathbb{F}_{p} \simeq Y_{p}^{r_{F}} \oplus Y_{p-1}$. In particular, the first claim of the proposition follows.

On the other hand, if $\mu_{F}(p) = \mu_{F_{1}}(p)$ and $\delta_{F}=1$, then we have $\hat{H}^{i}(\mathrm{Gal}(F_{1}/F), \mu_{F_{1}}(p)) \simeq \mathbb{F}_p$ for all $i \in \mathbb{Z}$. Therefore, $b$ is equal to $1$ or $2$ because it is positive by the Dirichlet-Herbrand theorem. Hence, $E_{F_1} \otimes_{\mathbb{Z}} \mathbb{F}_p$ is isomorphic to one of $Y^{r_F}_p \oplus Y_{p-1}$ or $Y^{r_F-1}_p \oplus Y^2_{p-1} \oplus \mathbb{F}_p$.

Suppose that $\mu_{F_1}(p)$ is non-trivial. Let $\{a_i \}_{1 \leq i \leq p}$ and $\{b_i \}_{1 \leq i \leq p}$ be the non-negative integers in the following Jordan normal forms
\begin{equation*}
    U_{F_1} \otimes_{\mathbb{Z}} \mathbb{F}_p \simeq \underset{i=1}{\overset{p}{\bigoplus}} \, Y^{a_i}_i \quad \text{and} \quad E_{F_1} \otimes_{\mathbb{Z}} \mathbb{F}_p \simeq \underset{i=1}{\overset{p}{\bigoplus}} \, Y^{b_i}_i.
\end{equation*}
We can use $(\ref{modpsequence})$ to study the relationship between $\{a_i\}_{1 \leq i \leq p}$ and $\{b_i\}_{1 \leq i \leq p}$. Since $\mu_{F_1} \otimes_{\mathbb{Z}} \mathbb{F}_p$ is isomorphic to $Y_1$, the $\mathbb{F}_p[X]$-module $\mu_{F_1} \otimes_{\mathbb{Z}} \mathbb{F}_p$ is contained in the kernel of the multiplication by $(X-1) \in \mathbb{F}_p[X]$ on $U_{F_1} \otimes_{\mathbb{Z}} \mathbb{F}_p$. Hence, we have the isomorphisms
\begin{equation*}
\underset{i=2}{\overset{p}{\bigoplus}} \, Y^{a_{i}}_{i-1} \simeq (X-1) \cdot (U_{F_1} \otimes_{\mathbb{Z}} \mathbb{F}_p) \simeq (X-1) \cdot (E_{F_1} \otimes_{\mathbb{Z}} \mathbb{F}_p) \simeq \underset{i=2}{\overset{p}{\bigoplus}} \, Y^{b_{i}}_{i-1}.
\end{equation*}
As a consequence, one has $a_1=b_1+1$ and $a_i=b_i$ for every $2 \leq i \leq p$. From this, the remaining claims of the proposition follow.
\end{proof}

To study $V_{\emptyset}(F_{1})/V_{\emptyset}(F)F_{1}^{\times \, p}$, we first need to know the image of $V_{\emptyset}(F)$ in $V_{\emptyset}(F_{1})/F_{1}^{\times \, p}$.

\begin{lemma}\label{imbedding}
The image of $V_{\emptyset}(F)$ in $V_{\emptyset}(F_{1})/F_{1}^{\times \, p}$ is generated by the image of the subgroup $U_{F} \subsetneq V_{\emptyset}(F)$ and an element of $U_{F_{1}}F_{1}^{\times \, p}/F_{1}^{\times \, p} \simeq U_{F_1} \otimes_{\mathbb{Z}} \mathbb{F}_p$ which is fixed by $\mathrm{Gal}(F_{1}/F)$. Therefore, the image of $V_{\emptyset}(F)$ in $V_{\emptyset}(F_{1})/F_{1}^{\times \,p}$ is contained in the submodule $(U_{F_{1}} \otimes_{\mathbb{Z}} \mathbb{F}_{p})^{\mathrm{Gal}(F_{1}/F)}$.
\end{lemma}

\begin{proof}
From the exact sequence for $F$ that is analogous to $(\ref{Shafarevichsequence})$, we can check that $V_{\emptyset}(F)$ is generated by $F^{\times \, p}$, $U_{F}$ and an element $\alpha \in F^{\times}$ such that $(\alpha) = \mathfrak{a}^{p} \in I_{F}$ for a fractional ideal $\mathfrak{a} \in I_{F}$ whose class in $\mathrm{Cl}_{F}$ is a generator of $\mathrm{Cl}_{F}[p] \simeq \mathbb{F}_p$. Since $\mathrm{Cl}_{F}[p]$ capitulates in $\mathrm{Cl}_{F_{1}}$, the ideal $\mathfrak{a}$ becomes principal in $I_{F_1}$, and we have $(\alpha) = (\beta)^{p}$ in $I_{F_{1}}$ for some $\beta \in F_{1}^{\times}$. Hence, $\alpha = \beta^{p}u$ for some $u \in U_{F_{1}}$ whose class in $V_{\emptyset}(F_{1})/F_{1}^{\times \,p}$ is fixed by $\mathrm{Gal}(F_{1}/F)$.
\end{proof}

\begin{proposition}\label{upperbound}
Let $p$ be an odd prime, and let $F$ be a number field with non-trivial cyclic $p$-class group. 
 Assume that for $p$=3, $F$ is not a complex quartic number field containing $\zeta_{3}$. Then the ratio $(\ref{ratio2})$ is bounded above by $p^{-\mathrm{max}\{r_{F}-1,1\}}$.
\end{proposition}

\begin{proof}
Suppose that $V_{\emptyset}(F_{1})/V_{\emptyset}(F)F_{1}^{\times \, p}$ is isomorphic to $\underset{i=1}{\overset{p}{\bigoplus}} Y_{i}^{t_{i}}$ as $\mathbb{F}_{p}[X]/(X^{p}-1)$-modules for some integers $\{t_{i}\}_{1 \leq i \leq p}$. Then, the ratio $(\ref{ratio2})$ is equal to $p^{-2t_{p}-t_{p-1}}$. Let $W$ be the image of $V_{\emptyset}(F)$ in $V_{\emptyset}(F_{1})/F_{1}^{\times \, p}$. By Lemma \ref{imbedding}, $W$ is imbedded in $U_{F_{1}} \otimes_{\mathbb{Z}} \mathbb{F}_{p} \subseteq V_{\emptyset}(F_1)/F^{\times \, p}_1$, and therefore, we have an exact sequence
\begin{equation*}
0 \longrightarrow (U_{F_{1}} \otimes_{\mathbb{Z}} \mathbb{F}_{p})/W \longrightarrow V_{\emptyset}(F_{1})/V_{\emptyset}(F)F_{1}^{\times \, p} \longrightarrow \mathrm{Cl}_{F_{1}}[p] \longrightarrow 0
\end{equation*}
of $\mathbb{F}_{p}[X]/(X^{p}-1)$-modules.
By Lemma \ref{imbedding}, $(U_{F_{1}} \otimes_{\mathbb{Z}} \mathbb{F}_{p})/W$ has $(U_{F_{1}} \otimes_{\mathbb{Z}} \mathbb{F}_{p})/(U_{F_{1}} \otimes_{\mathbb{Z}} \mathbb{F}_{p})^{\mathrm{Gal}(F_{1}/F)}$ as a quotient. By Proposition \ref{Galoismoduleunit}, $(U_{F_{1}} \otimes_{\mathbb{Z}} \mathbb{F}_{p})/(U_{F_{1}} \otimes_{\mathbb{Z}} \mathbb{F}_{p})^{\mathrm{Gal}(F_{1}/F)}$ has $Y_{p-1}^{r_{F}-1}$ as a direct factor. Hence, if $\underset{i=1}{\overset{p}{\bigoplus}} \, Y^{c_i}_i$ is the Jordan normal form of $(U_{F_1} \otimes_{\mathbb{Z}} \mathbb{F}_p)/W$, then we have the inequalities
\begin{equation*}
r_F -1 \leq 2c_p + c_{p-1} \leq 2t_p + t_{p-1}.
\end{equation*}
The first inequality follows from the existence of an epimorphism
\begin{equation*}
Y^{c_p}_2 \oplus \mathbb{F}^{c_{p-1}}_p \simeq \Psi(X) \cdot \big ( (U_{F_{1}} \otimes_{\mathbb{Z}} \mathbb{F}_p ) / W \big ) \longrightarrow \Psi(X) \cdot Y^{r_F-1}_{p-1} \simeq \mathbb{F}^{r_F-1}_p,
\end{equation*}
and the second inequality follows from the inclusion
\begin{equation*}
    \Psi(X) \cdot \big ( (U_{F_{1}} \otimes_{\mathbb{Z}} \mathbb{F}_p ) / W \big ) \subseteq \Psi(X) \cdot V_{\emptyset}(F_1)/V_{\emptyset}(F)F^{\times \, p}_1 \simeq Y^{t_p}_2 \oplus \mathbb{F}^{t_{p-1}}_p.
\end{equation*}

If $r_{F}=1$, then we have $\delta_{F}=0$ unless $p=3$ and $F$ is a quadratic extension of $\mathbb{Q}(\zeta_{3})$. If $\delta_{F}=0$ and $r_{F}=1$, then we have $U_{F_{1}} \otimes_{\mathbb{Z}} \mathbb{F}_{p} \simeq Y_{p} \oplus Y_{p-1}$ by Proposition \ref{Galoismoduleunit}. Therefore, $U_{F_1} \otimes_{\mathbb{Z}} \mathbb{F}_p$ has $Y_p$ as a quotient. By the same argument, we have $2t_{p} + t_{p-1} \geq 1$, and the claim follows.
\end{proof}

Before giving a proof of Theorem \ref{maintheorem2}, let us make a following lemma.

\begin{lemma}\label{invariance}
The set $\mathrm{Gal}(\mathrm{Gov}(F_{1})/\mathrm{Gov}(F)F_{1})[\Psi]$ is invariant under the conjugation action of $\mathrm{Gal}(\mathrm{Gov}(F)F_{1}/F)$.
\end{lemma}

\begin{proof}
Since $F_{1}(\zeta_{p})/F$ is an abelian extension, for any $g \in \mathrm{Gal}(\mathrm{Gov}(F)F_{1}/F)$ and $\gamma \in \mathrm{Gal}(F_{1}(\zeta_{p})/F(\zeta_{p}))$, the restriction of $g$ to $F_{1}(\zeta_{p})$ commutes with $\gamma$. Therefore, for any $x \in \mathrm{Gal}(\mathrm{Gov}(F_{1})/\mathrm{Gov}(F)F_{1})$ and some liftings $\tilde{\gamma}, \tilde{g}$ of $\gamma$ and $g$ to $\mathrm{Gov}(F_{1})$, we have
\begin{equation*}
\tilde{g}^{-1}\tilde{\gamma} \tilde{g} x \tilde{g}^{-1}\tilde{\gamma}^{-1}\tilde{g} = \tilde{\gamma} x \tilde{\gamma}^{-1}.
\end{equation*}
Thus, the conjugation action of $g$ on $\mathrm{Gal}(\mathrm{Gov}(F_{1})/\mathrm{Gov}(F)F_{1})$ is equivariant over $\mathrm{Gal}(F_{1}/F)$.
\end{proof}

\vskip 10pt

\begin{proof}[\textbf{Proof of Theorem \ref{maintheorem2}}]
Let $\mathfrak{M}$ and $\mathfrak{M}'$ be the sets of primes of $F$ in the statement of Theorem \ref{maintheorem2}. By the remark at the beginning of \S \ref{subsubsection4.2.1}, $\mathfrak{M}$ is equal to the set of non-$p$-adic primes $\mathfrak{q}$ of $F$ which split completely in $\mathrm{Gov}(F)F_{1}$. By Proposition \ref{mainproposition}, $\mathfrak{M}'$ is a subset of the set $\mathfrak{M}''$ of primes $\mathfrak{q} \in \mathfrak{M}$ such that for each prime $\mathfrak{P}$ of $\mathrm{Gov}(F)F_{1}$ above $\mathfrak{q}$, the Frobenius automorphism of $\mathrm{Gal}(\mathrm{Gov}(F_{1})/\mathrm{Gov}(F)F_{1})$ at $\mathfrak{P}$ is annihilated by $\Psi(X)$. Let $\mathfrak{N}$ be the set of primes of $\mathrm{Gov}(F)F_{1}$, and let $\mathfrak{N}'$ be the set of non-$p$-adic primes of $\mathrm{Gov}(F)F_{1}$ whose Frobenius automorphism in $\mathrm{Gal}(\mathrm{Gov}(F_{1})/\mathrm{Gov}(F)F_{1})$ is annihilated by $\Psi(X)$. By the Chebotarev density theorem, we have
\begin{equation*}
\frac{\sharp \, \mathrm{Gal}(\mathrm{Gov}(F_{1})/\mathrm{Gov}(F)F_{1})[\Psi]}{[\mathrm{Gov}(F_{1}):\mathrm{Gov}(F)F_{1}]} = \underset{s \to 1^{+}}{\lim} \, \frac{\,\,\,\,\,\underset{\mathfrak{p} \in \mathfrak{N}'}{\sum}\frac{1}{\mathrm{N}(\mathfrak{p})^{s}}\,\,\,\,\,}{\underset{\mathfrak{p} \in \mathfrak{N}}{\sum} \frac{1}{\mathrm{N}(\mathfrak{p})^{s}}}.
\end{equation*}
By Lemma \ref{invariance}, $\mathfrak{N}'$ is invariant under the action of $\mathrm{Gal}(\mathrm{Gov}(F)F_{1}/F)$. Let $D(\mathfrak{M}'')$ be the Dirichlet density of $\mathfrak{M}''$.
Since the set $\mathfrak{N}_0$ of primes of $\mathrm{Gov}(F)F_{1}$ with absolute residue class degree larger than $1$ has Dirichlet density $0$, we have
\begin{equation*}
\underset{s \to 1^{+}}{\lim} \, \frac{\,\,\,\,\,\underset{\mathfrak{p} \in \mathfrak{N}'}{\sum}\frac{1}{\mathrm{N}(\mathfrak{p})^{s}}\,\,\,\,\,}{\underset{\mathfrak{p} \in \mathfrak{N}}{\sum} \frac{1}{\mathrm{N}(\mathfrak{p})^{s}}} = \underset{s \to 1^{+}}{\lim} \, \frac{\,\,\,\,\,\underset{\mathfrak{p} \in \mathfrak{N}' \setminus \mathfrak{N}_0}{\sum}\frac{1}{\mathrm{N}(\mathfrak{p})^{s}}\,\,\,\,\,}{\underset{\mathfrak{p} \in \mathfrak{N} \setminus \mathfrak{N}_0}{\sum} \frac{1}{\mathrm{N}(\mathfrak{p})^{s}}} = \underset{s \to 1^{+}}{\lim} \, \frac{\,\,\,\,\,\underset{\mathfrak{q} \in \mathfrak{M}''}{\sum}\frac{1}{\mathrm{N}(\mathfrak{q})^{s}}\,\,\,\,\,}{\underset{\mathfrak{q} \in \mathfrak{M}}{\sum} \frac{1}{\mathrm{N}(\mathfrak{q})^{s}}} = \frac{D(\mathfrak{M}'')}{D(\mathfrak{M})},
\end{equation*}
where the second equality follows from Lemma \ref{invariance}. Therefore, we have
\begin{equation*}
\frac{D(\mathfrak{M}')}{D(\mathfrak{M})} \leq \frac{D(\mathfrak{M}'')}{D(\mathfrak{M})} = \frac{\sharp \, \mathrm{Gal}(\mathrm{Gov}(F_{1})/\mathrm{Gov}(F)F_{1})[\Psi]}{[\mathrm{Gov}(F_{1}):\mathrm{Gov}(F)F_{1}]}.
\end{equation*}
The theorem follows from Proposition \ref{upperbound}.
\end{proof}

To summarize, by the remark at the beginning of \S \ref{subsubsection4.2.1}, $G_{\{\mathfrak{q}\}}(F)$ is finite unless $\mathfrak{q}$ splits completely in $\mathrm{Gov}(F)F_{1}$. On the other hand, unless $p=3$ and $F$ is a complex quartic number field containing $\zeta_{3}$, $G_{\{\mathfrak{q}\}}(F)$ is finite for the majority of $\mathfrak{q}$'s which split completely in $\mathrm{Gov}(F)F_1$ by Theorem \ref{maintheorem2}. This conclusion leads to the following corollary.

\begin{corollary}
Let $F$ be a number field with a non-trivial cyclic $p$-class group. Assume that for $p=3$, $F$ is not a complex quartic number field containing $\zeta_{3}$. Then the Dirichlet density of the set of primes $\mathfrak{q}$ of $F$ for which $G_{\{\mathfrak{q}\}}(F)$ is infinite is bounded above by $[\mathrm{Gov}(F)F_{1}:F]^{-1} \cdot p^{-\mathrm{max}\{r_{F}-1,1\}}$, which is at most
\begin{equation*}
\frac{1}{[F(\zeta_{p}):F] \cdot p^{r_{F}+1+\delta_{F} + \mathrm{max}\{r_{F}-1,1\}}}.
\end{equation*}

\end{corollary}
\begin{proof}
We have $D(\mathfrak{M}') \leq D(\mathfrak{M}) \cdot p^{-\mathrm{max}\{r_{F}-1,1\}}$. By the Chebotarev density theorem, we have $D(\mathfrak{M}) = 1/[\mathrm{Gov}(F)F_{1}:F]$. The corollary thus follows from the inequality \begin{equation*}
[\mathrm{Gov}(F)F_{1}:F] = [F(\zeta_{p}) : F] \cdot [\mathrm{Gov}(F)F_{1} : F(\zeta_{p})] \geq [F(\zeta_{p}):F] \cdot p^{r_{F}+\delta_{F}+1}.
\end{equation*}
\end{proof}

\noindent \textbf{Acknowledgements}
\vskip 10pt
We would like to thank Christian Maire for his many helpful comments and for reading the manuscript carefully. We also would like to thank Zakariae Bouazzaoui, Oussama Hamza, Youness Mazigh, and Ali Mouhib for helpful discussions. Lastly, we would like to thank the anonymous reviewer for many helpful comments which improved the manuscript.

The authors were supported by the Core Research Institute Basic Science Research Program through the National Research Foundation of Korea(NRF) funded by the Ministry of Education(Grant No. 2019R1A6A1A11051177).
Y. Lee was also supported by the National Research Foundation of Korea(NRF) grant funded by the Korea government (MSIT)(NRF- 2022R1A2C1003203). D. Lim was also supported by the Basic Science Research Program through the National Research Foundation of Korea (NRF) funded by the Ministry of Education (Grant No. NRF-2022R1I1A1A01071431).

\bibliographystyle{elsarticle-harv} 
\bibliography{finitude.bib}

\end{document}